\documentclass[12pt,letterpaper]{article}
\usepackage[T1]{fontenc}
\usepackage[english]{babel}	
\usepackage{german}

\usepackage{graphicx}
\usepackage[a4paper]{geometry}
\usepackage{fullpage}
\usepackage{hyperref}

\usepackage{algorithm}
\usepackage{algpseudocode} 
\usepackage{amsmath}
\usepackage{amsthm}
\usepackage{amssymb}
\usepackage{mathpazo}
\usepackage{array}

\theoremstyle{definition}
\newtheorem{defi}{Definition}[section]
\newtheorem{rema}{Remark}[section]
\newtheorem{exam}{Example}[section]

\theoremstyle{plain}
\newtheorem{theo}{Theorem}[section]
\newtheorem{lemm}[theo]{Lemma}
\newtheorem{prop}[theo]{Proposition}
\newtheorem{coro}[theo]{Corollary}

\numberwithin{equation}{section}

\newcommand{\mN}{\mathbb{N}}
\newcommand{\mZ}{\mathbb{Z}}
\newcommand{\mP}{\mathbb{P}}
\newcommand{\mR}{\mathbb{R}}
\newcommand{\mdot}{\!\cdot\!}
\newcommand{\ndiv}{\!\nmid\!}
\newcommand{\mdiv}{\!\mid\!}
\newcommand{\copr}{\!\perp\!}
\newcommand{\ncopr}{\!\not\perp\!}
\newcommand{\vl}{\vrule width 2pt}

\begin{document}
%%%%%%%%%%%%%%%%%%%%%%%%%%%%%%%%%%%%%%%%%%%%%%%%%%%%%%%

%%%%%%%%%%%%%%%%%%%%%%%%%%%%%%%%%%%%%%%%%%%%%%%%%%%%%%%
%%%%%%%%%%%%%%%%%%%%%%%%%%%%%%%%%%%%%%%%%%%%%%%%%%%%%%%
\title{Algorithmic concepts for the computation of Jacobsthal's function}
\author{Mario Ziller and John F. Morack}
\date{}
%%%%%%%%%%%%%%%%%%%%%%%%%%%%%%%%%%%%%%%%%%%%%%%%%%%%%%%
%%%%%%%%%%%%%%%%%%%%%%%%%%%%%%%%%%%%%%%%%%%%%%%%%%%%%%%

\maketitle

%%%%%%%%%%%%%%%%%%%%%%%%%%%%%%%%%%%%%%%%%%%%%%%%%%%%%%%
%%%%%%%%%%%%%%%%%%%%%%%%%%%%%%%%%%%%%%%%%%%%%%%%%%%%%%%
\begin{abstract}
%%%%%%%%%%%%%%%%%%%%%%%%%%%%%%%%%%%%%%%%%%%%%%%%%%%%%%%
%%%%%%%%%%%%%%%%%%%%%%%%%%%%%%%%%%%%%%%%%%%%%%%%%%%%%%%

The Jacobsthal function has aroused interest in various contexts in the past decades. We review several algorithmic ideas for the computation of Jacobsthal's function for primorial numbers and discuss their practicability regarding computational effort. The respective function values were computed for primes up to 251. In addition to the results including  previously unknown data, we provide exhaustive lists of all sequences of the appropriate maximum lengths in ancillary files.\\
\end{abstract}

%%%%%%%%%%%%%%%%%%%%%%%%%%%%%%%%%%%%%%%%%%%%%%%%%%%%%%%
%%%%%%%%%%%%%%%%%%%%%%%%%%%%%%%%%%%%%%%%%%%%%%%%%%%%%%%
\section{Introduction}
%%%%%%%%%%%%%%%%%%%%%%%%%%%%%%%%%%%%%%%%%%%%%%%%%%%%%%%
%%%%%%%%%%%%%%%%%%%%%%%%%%%%%%%%%%%%%%%%%%%%%%%%%%%%%%%

Henceforth, we denote the set of integral numbers by $\mZ$ and the set of natural numbers, i.e. positive integers, by $\mN$. $\mP=\{p_i\mid i\in\mN\}$ is the set of prime numbers with $p_1=2$. As usual, we define the $n^{th}$ primorial number as the product of the first  $n$ primes: $p_n\#=\prod_{i=1}^n p_i\ , n\in\mN$. We essentially follow the notation of Hagedorn 2009  \cite{Hagedorn_2009}.\\

The ordinary Jacobsthal function $j(n)$ is defined to be the smallest positive integer $m$, such that every sequence of $m$ consecutive integers contains at least one integer coprime to $n$ \cite{Jacobsthal_1960_I, Erdos_1962, Hagedorn_2009, Costello_Watts_2013}.

\begin{defi} \label{Jacobsthal} {\itshape Jacobsthal function.} \cite{Sequence_A048669}\\
For $n\in\mN$, the Jacobsthal function $j(n)$ is defined as
$$j(n)=\min\ \{m\in\mN\mid\forall\ a\in\mZ\ \exists\ q\in\{1,\dots,m\}:a+q\copr n\}.$$
\end{defi}

This definition is equivalent to the formulation that $j(n)$ is the greatest difference $m$ between two terms in the sequence of integers which are coprime to $n$.
\begin{equation*} \begin{split}
j(n)=\max\ \{m\in\mN\mid\ &\exists\ a\in\mZ\ :a\copr n\land a+m\copr n\ \land\\
&\forall\ q\in\{1,\dots,m-1\}:a+q\ncopr n\}.
\end{split} \end{equation*}

In other words, $(j(n)-1)$ is the greatest length $m^*=m-1$ of a sequence of consecutive integers which are not coprime to $n$. For this reason, we define a reduced variant of the Jacobsthal function \cite{Jacobsthal_1960_I}. This will make a simplified representation possible.

\begin{defi} \label{reduced} {\itshape Reduced Jacobsthal function.} \cite{Sequence_A132468}\\
For $n\in\mN$, the reduced Jacobsthal function $j^*(n)$ is defined as
$$j^*(n)=j(n)-1=\max\ \{m^*\in\mN\mid\exists\ a\in\mZ\ \forall\ q\in\{1,\dots,m^*\}:a+q\ncopr n\}.$$
\end{defi}

\begin{rema} \label{properties}
The following statements are elementary consequences of the definition of Jacobsthal's function and describe some interesting properties of it \cite{Jacobsthal_1960_I}.

Product.\\
\hspace*{ 1cm}$\forall\ n1,n2 \in\mN:j(n1\mdot n2)\ge j(n1) \land j(n1\mdot n2)\ge j(n2)$.

Coprime product.\\
\hspace*{ 1cm}$\forall\ n1,n2 \in\mN>1\mid n1\copr n2:j(n1\mdot n2)>j(n1)\land j(n1\mdot n2)>j(n2)$.

Greatest common divisor.\\
\hspace*{ 1cm}$\forall\ n1,n2 \in\mN:j(gcd(n1,n2))\le j(n1) \land j(gcd(n1,n2))\le j(n2)$.

Prime power.\\
\hspace*{ 1cm}$\forall\ n,k \in\mN\ \forall\ p \in\mP:j(p^k\mdot n)=j(p\mdot n)$.

Prime separation.\\
\hspace*{ 1cm}$\forall\ n,n^*,k \in\mN\ \forall\ p \in\mP\mid n=p^k\mdot n^*,p\copr n^*:j(n)=j(p\mdot n^*)$.
\end{rema}

The last remark implies that the entire Jacobsthal function is determined by its \linebreak values for products of distinct primes \cite{Jacobsthal_1960_I}. In his subsequent elaborations \cite{Jacobsthal_1960_I,Jacobsthal_1960_II,Jacobsthal_1960_III,Jacobsthal_1961_IV,Jacobsthal_1961_V}, Jacobsthal derived explicit formulae for the calculation of function values for squarefree integers containing up to 7 distinct prime factors, and bounds of the function for up to 10 distinct prime factors.

The particular case of primorial numbers is therefore most interesting because the function values at these points contain the relevant information for constructing general upper bounds. The Jacobsthal function of primorial numbers $h(n)$ \cite{Hagedorn_2009} is therefore defined as the smallest positive integer $m$, such that every sequence of $m$ consecutive integers contains an integer coprime to the product of the first n primes.

\begin{defi} {\itshape Primorial Jacobsthal function.} \label{h} \cite{Sequence_A048670}\\
For $n\in\mN$, the primorial Jacobsthal function $h(n)$ is defined as
$$h(n)=j(p_n\#).$$
\end{defi}

By analogy with definition \ref{reduced}, we also define a reduced variant of the latter function which represents the greatest length of a sequence of consecutive integers which are not coprime to the $n^{th}$ primorial number.

\begin{defi} {\itshape Reduced primorial Jacobsthal function.} \cite{Sequence_A058989}\\
For $n\in\mN$, the reduced primorial Jacobsthal function $h^*(n)$ is defined as
\begin{equation*} \begin{split}
&h^*(n)=h(n)-1,\\
&h^*(n)=\max\ \{m^*\in\mN\mid\exists\ a\in\mZ\ \forall\ q\in\{1,\dots,m^*\}:a+q\ncopr p_n\#\}.
\end{split} \end{equation*}
\end{defi}

For the effective computation of $h(n)$, or $h^*(n)$ respectively, it is sufficient to omit the first prime 2 from the calculation. Therefore, we define a condensed Jacobsthal function $\omega(n)$ which is directly related to $h(n)$. Its computation reduces unnecessary effort. Hagedorn described this function in the context of killing sieves \cite{Gordon_Rodemich_1998,  Hagedorn_2009} which we avoid here. We prefer the straightforward derivation of it for our purpose. Furthermore, we harmonise the arguments of the functions $\omega(n)$ and $h(n)$ so that $n$ in both of them refer to the greatest considered prime $p_n$ whereas Hagedorn \cite{Hagedorn_2009} links $\omega(n)$ to $p_{n+1}$.

\begin{defi} {\itshape Condensed Jacobsthal function.} \label{omega} \cite{Sequence_A072752}\\
For $n\in\mN>1$, the condensed Jacobsthal function $\omega(n)$ is defined as the greatest length of a sequence of consecutive integers which are not coprime to the product of the odd primes through $p_{n}$.
$$\omega(n)=j^*(p_{n}\# /2).$$
\end{defi}

In other words, $\omega(n)$ is the greatest length of a sequence of consecutive integers each of which is divisible by one of the odd primes through $p_{n}$.\\

The prime 2 plays a specific role for the Jacobsthal function. The following lemma describes an interesting property of it and makes the direct calculation of $j(n)$ as a function of $\omega(n)$ possible.

\begin{lemm} \label{even}
Let $n\in\mN$ with $2\ndiv n$. Then\\
$$j(2\mdot n)=2 \mdot j(n).$$
\end{lemm}

\begin{proof}
We get $j(1)=1$ because every natural number is coprime to 1. Furthermore, $j(2)=2$ because every sequence of two consecutive integers containes one even and one odd number.

Let now n>1. According to definition 1.2, we resume ...
%According to definition \ref{reduced}, we resume
$$j(n)=\max\ \{m\in\mN\mid\exists\ a\in\mZ\ \forall\ q\in\{1,\dots,m-1\}:a+q\ncopr n\}.$$
Given $m\in\mN$ and $a\in\mZ$ as above. Then, there exists an $a^*\in\mZ$ with
$a^*\equiv 2\mdot a\ (mod\ n)$ and $a^*\equiv 1\ (mod\ 2)$ due to the Chinese remainder theorem.
With this $a^*$, we get for $q=1,\dots,m-1$
$$a^*+2\mdot q\equiv 2\mdot a+2\mdot q\equiv 2\mdot(a+q)\ (mod\ n),\ \text{and}$$
$$a^*+2\mdot q-1\equiv 1+2\mdot q-1\equiv 0\ (mod\ 2).$$

Because in addition $a^*+2\mdot m-1\equiv 0\ (mod\ 2)$,
$$\exists\ a^*\in\mZ\ \forall\ q\in\{1,\dots,2\mdot m-1\}:a^*+q\ncopr 2\mdot n\}$$
holds, and $j(2\mdot n)\ge 2\mdot m=2\mdot j(n)$ follows. The maximality remains retained in both directions.
\end{proof}

\begin{coro} \label{h_omega}
Let $n\in\mN>1$.\\
$$h(n)=2\mdot\omega(n)+2.$$
\end{coro}

\begin{proof}\ 

$\omega(n)=j^*(p_{n}\# /2)=j(p_{n}\# /2)-1=j(p_{n}\#)/2-1.$

$h(n)=j(p_n\#)=2\mdot\omega(n)+2.$
\end{proof}

The function $\omega(n)$ is only defined for $n>1$. For $n=1$, we remark $h(1)=j(2)=2$.\\

The counting of $\omega(n)$ remains the central computational problem. By definition, $\omega(n)$ is the greatest length of a sequence of consecutive integers which are not coprime to all of the first $n-1$ odd primes $p_2,\dots,p_n$. This means that every integer of the sequence must be divisible by at least one of these primes. On the other hand, those sequences can be characterised by a unique choice of non-zero residue classes for each prime.

\begin{prop} \label{remainders}
Let $m,n\in\mN,\ n>1$.\\ The following two statements are equivalent if $m=\omega(n)$.
  \begin{enumerate}
\item[(1)]
There is an $a\in\mZ$ so that every integer of the sequence $\{a+1,\dots,a+m\}$\\
is divisible by one of the primes $p_2,\dots,p_n$.
$$\exists\ a\in\mZ\ \forall\ q\in\{1,\dots,m\}\ \exists\ i\in\{2,\dots,n\}:a+q\equiv 0\ (mod\ p_i).$$
\item[(2)]
For every prime $p_2,\dots,p_n$, there exists one non-zero residue class\\
so that every integer of the sequence $\{1,\dots,m\}$ belongs to one of them.
\begin{equation*} \begin{split}
&\exists\ a_i\in\{1,\dots,p_i-1\},\ i=2,\dots,n\\
&\forall\ q\in\{1,\dots,m\}\ \exists\ i\in\{2,\dots,n\}:q\equiv a_i\ (mod\ p_i).
\end{split} \end{equation*}
  \end{enumerate}
\end{prop}

\begin{proof}\ \\
$(1)\Rightarrow(2)$:

\mbox{$a_i\equiv-a\ (mod\ p_i)$}, \mbox{$i=2,\dots,n$} satisfy the respective congruences of (2).\linebreak $p_i\ndiv a$ because m is maximum.\\
$(2)\Rightarrow(1)$:

According to the Chinese remainder theorem, there exists an $a\in\mZ$ solving the system of simultaneous congruences $a\equiv-a_i\ (mod\ p_i),\ i=2,\dots,n$. With this solution $a$, $\{a+1,\dots,a+m\}$ fulfils (1).
\end{proof}

\begin{rema}
In this proposition, $m$ need not necessarily be the maximum sequence length. It remains true for $m\le\omega(n)$ if $p_i\ndiv a$ for all $i=2,\dots,n$. Furthermore, the proposition holds for any set of distinct primes. The specific choice of the primes was not used in the proof.
\end{rema}

\begin{coro}
For every sequence of maximum length satisfying proposition \ref{remainders} (2), there exists a reverse sequence with
\begin{equation*} \begin{split}
&\exists\ b_i\in\{1,\dots,p_i-1\},\ i=2,\dots,n\\
&\forall\ q\in\{1,\dots,m\}\ \exists\ i\in\{2,\dots,n\}:m+1-q\equiv b_i\ (mod\ p_i).
\end{split} \end{equation*}
\end{coro}

\begin{proof}
The maximality of the given sequence length implies $m+1\not\equiv a_i\ (mod\ p_i)$ \linebreak for all $i=2,\dots,n$. Therefore, $b_i\equiv m+1-a_i\ (mod\ p_i),\ i=2,\dots,n$ satisfy the \linebreak requirements.
\end{proof}

The pairs of reverse sequences define a symmetry within the set of sequences of maximum length. The algorithmic exploitation of this interesting feature, however, seems to be difficult because $m$ is a priori unknown.\\

There is another way to characterise a sequence with the considered properties. Given a permutation of the primes $\{p_2,\dots,p_n\}$, the sequence can be constructed from left to right covering the next free position recursively. Conversely, a permutation of this kind can be derived from a given sequence.

\begin{prop} \label{permutation}
Let $m,n\in\mN,\ n>1$.\\ The following two statements are equivalent if $m=\omega(n)$.
  \begin{enumerate}
\item[(2)]
For every prime $p_2,\dots,p_n$, there exists one non-zero residue class\\
so that every integer of the sequence $\{1,\dots,m\}$ belongs to one of them.
\begin{equation*} \begin{split}
&\exists\ a_i\in\{1,\dots,p_i-1\},\ i=2,\dots,n\\
&\forall\ q\in\{1,\dots,m\}\ \exists\ i\in\{2,\dots,n\}:q\equiv a_i\ (mod\ p_i).
\end{split} \end{equation*}
\item[(3)]
There exists a permutation $(\pi_2,\dots,\pi_n)$ of $\{p_2,\dots,p_n\}$ and a tuple $(q_2,\dots,q_n)$\\
so that the sequence $\{1,\dots,m\}$ is completely covered by the residue classes $q_i$ mod $\pi_i$ when all $\pi_i$ were recursively assigned to the first free position $q_i$, respectively.
\begin{equation*} \begin{split}
&\exists\ \pi_i\in\{p_2,\dots,p_n\}\ \exists\ q_i\in\{1,\dots,m\},\ i=2,\dots,n\\
&\quad with\ \{\pi_2,\dots,\pi_n\}\setminus\{p_2,\dots,p_n\}=\emptyset,\\
&\quad q_i\not\equiv 0\ (mod\ \pi_i),\ i=2,\dots,n,\\
&\quad i<j \Rightarrow q_i<q_j,\ i,j\in\{2,\dots,n\},\\
&\quad q_2=1,\ and\\
&\quad q_i=min\{j\in\{1,\dots,m\}\mid \forall k<i : j\not\equiv q_k\ (mod\ \pi_k)\},\ i=3,\dots,n:\\
&\forall\ q\in\{1,\dots,m\}\ \exists\ i\in\{2,\dots,n\}:q\equiv q_i\ (mod\ \pi_i).
\end{split} \end{equation*}
  \end{enumerate}
\end{prop}

\begin{proof}\ \\
$(2)\Rightarrow(3)$:

We set $q_2=1$ and choose the smallest $p_j$ with $a_j\equiv 1\ (mod\ p_j)$ as $\pi_2$. Given $q_k$ and $\pi_k$ for $2\le k<i\le n$, we set $q_i=min\{j\in\{1,\dots,m\}\mid \forall k<i : j\not\equiv q_k\ (mod\ \pi_k)\}$. Then we choose $\pi_i$ as the smallest $p_j$ with $a_j\equiv q_i\ (mod\ p_j)$. There must exist a proper $p_j$ because $m=\omega(n)$. Otherwise, at least $p_j$ was not needed to cover a sequence of length $m$ and this sequence could be extended to a longer one by setting $q_i=m+1$ and  $\pi_i=p_j$ which contradicts to the definition of $\omega(n)$. All $q_i$ and $\pi_i$ inductively selected as described above fulfils (3).\\
$(3)\Rightarrow(2)$:

For every $i\in\{2,\dots,n\}$ , there exists a unique $j\in\{2,\dots,n\}$ with $\pi_j=p_i$ because $(\pi_2,\dots,\pi_n)$ is a permutation of $\{p_2,\dots,p_n\}$.
$a_i\equiv q_j\ (mod\ p_i),\ i=2,\dots,n$ satisfies the respective congruences of (2).
\end{proof}

\begin{rema}
The set of sequences with the property (2) includes any sequence fulfilling (3), even if $m<\omega(n)$. The proof did not make use of that condition. The proposition again holds for any set of distinct primes. The specific choice of the prime set was also not used in the proof.
\end{rema}

\begin{exam}
We give an example for $n=6$ for all of the three variants of the propositions \ref{remainders} and \ref{permutation}. The primes to be considered are  3, 5, 7, 11, and 13. The result $\omega(n)$=10 is reached in a sequence with the following properties:\\
Prime sequence. Proposition \ref{remainders} (1).\\
\hspace*{ 1cm}$a=12227$.\\
\hspace*{ 1cm}$3/(a+1),\ 7/(a+2),\ 5/(a+3),\ 3/(a+4),\ 11/(a+5),$\\
\hspace*{ 1cm}$13/(a+6),\ 3/(a+7),\ 5/(a+8),\ 7/(a+9),\ 3/(a+10)$.\\
Set of remainders. Proposition \ref{remainders} (2) or proposition \ref{permutation} (2).\\
\hspace*{ 1cm}$a_2=1,\ a_3=3,\ a_4=2,\ a_5=5,\ a_6=6$.\\
\hspace*{ 1cm}$1\equiv 1\ (mod\ 3),\ 2\equiv 2\ (mod\ 7),\ 3\equiv 3\ (mod\ 5),\ 4\equiv 1\ (mod\ 3),
                            \ 5\equiv 5\ (mod\ 11),$\\
\hspace*{ 1cm}$6\equiv 6\ (mod\ 13),\ 7\equiv 1\ (mod\ 3),\ 8\equiv 3\ (mod\ 5),\ 9\equiv 2\ (mod\ 7),
                            \ 10\equiv 1\ (mod\ 3)$.\\
Prime Permutation. Proposition \ref{permutation} (3).\\
\hspace*{ 1cm}$\pi_2=3,\ \pi_3=7,\ \pi_4=5,\ \pi_5=11,\ \pi_6=13$.\\
\hspace*{ 1cm}$q_2=1,\ q_3=2,\ q_4=3,\ q_5=5,\ q_6=6$.\\
\hspace*{ 1cm}$1\equiv 1\ (mod\ 3),\ 2\equiv 2\ (mod\ 7),\ 3\equiv 3\ (mod\ 5),\ 4\equiv 1\ (mod\ 3),
                            \ 5\equiv 5\ (mod\ 11),$\\
\hspace*{ 1cm}$6\equiv 6\ (mod\ 13),\ 7\equiv 1\ (mod\ 3),\ 8\equiv 3\ (mod\ 5),\ 9\equiv 2\ (mod\ 7),
                            \ 10\equiv 1\ (mod\ 3)$.
\end{exam}

\begin{rema} \label{representation}
As a consequence of the propositions \ref{remainders} and \ref{permutation}, we can formulate three equivalent descriptions of the condensed Jacobsthal function $\omega(n)$ for $n>1$.
  \begin{enumerate}
\item[(1)]
The function $\omega(n)$ is the maximum length $m$ of a sequence of consecutive \linebreak integers where each of them is divisible  by at least one of the first $n$ odd primes $p_2,\dots,p_n$.
\begin{equation*} \begin{split}
\omega(n)=\max\ \{&m\in\mN\mid\exists\ a\in\mZ\\
&\forall\ q\in\{1,\dots,m\}\ \exists\ i\in\{2,\dots,n\}:a+q\equiv 0\ (mod\ p_i)\}.
\end{split} \end{equation*}
\item[(2)]
The function $\omega(n)$ is the maximum $m\in\mN$ for which there exists a set of non-zero remainders $a_i\ mod\ p_i,\ i=2,\dots,n$ so that every $q\in\{1,\dots,m\}$ satisfies one of the congruences $q\equiv a_i\ (mod\ p_i)$.
\begin{equation*} \begin{split}
\omega(n)=\max\ \{&m\in\mN\mid\exists\ a_i\in\{1,\dots,p_i-1\},\ i=2,\dots,n\\
&\forall\ q\in\{1,\dots,m\}\ \exists\ i\in\{2,\dots,n\}:q\equiv a_i\ (mod\ p_i)\}.
\end{split} \end{equation*}
\item[(3)]
The function $\omega(n)$ is the maximum $m\in\mN$ for which there exist a permutation $(\pi_2,\dots,\pi_n)$ of $\{p_2,\dots,p_n\}$ and a tuple $(q_2,\dots,q_n)$ so that the sequence $\{1,\dots,m\}$ is completely covered by the residue classes $q_i$ mod $\pi_i$ when all $\pi_i$ were recursively assigned to the first free position $q_i$, respectively.
\begin{equation*} \begin{split}
\omega(n)=\max\ \{&m\in\mN\mid\exists\ \pi_i\in\{p_2,\dots,p_n\}\ \exists\ q_i\in\{1,\dots,m\},\ i=2,\dots,n\\
&\quad with\ \{\pi_2,\dots,\pi_n\}\setminus\{p_2,\dots,p_n\}=\emptyset,\\
&\quad q_i\not\equiv 0\ (mod\ \pi_i),\ i=2,\dots,n,\\
&\quad i<j \Rightarrow q_i<q_j,\ i,j\in\{2,\dots,n\},\\
&\quad q_2=1,\ and\ for\ i=3,\dots,n\\
&\quad q_i=min\{j\in\{1,\dots,m\}\mid \forall k<i : j\not\equiv q_k\ (mod\ \pi_k)\}:\\
&\forall\ q\in\{1,\dots,m\}\ \exists\ i\in\{2,\dots,n\}:q\equiv q_i\ (mod\ \pi_i)\}.
\end{split} \end{equation*}
  \end{enumerate}
\end{rema}
\ \\

In the next chapter, we will review several algorithmic ideas for the computation of Jacobsthal's function for primorial numbers. All of these approaches utilise the\linebreak versions (2) or (3) of understanding $\omega(n)$. When we present our results below, we will get back to the recent remark in order to derive demonstrative forms for the presentation of sequences.\\

%%%%%%%%%%%%%%%%%%%%%%%%%%%%%%%%%%%%%%%%%%%%%%%%%%%%%%%
%%%%%%%%%%%%%%%%%%%%%%%%%%%%%%%%%%%%%%%%%%%%%%%%%%%%%%%
\section{Computation of $\omega(n)$}
\label{Algorithms}
%%%%%%%%%%%%%%%%%%%%%%%%%%%%%%%%%%%%%%%%%%%%%%%%%%%%%%%
%%%%%%%%%%%%%%%%%%%%%%%%%%%%%%%%%%%%%%%%%%%%%%%%%%%%%%%

%%%%%%%%%%%%%%%%%%%%%%%%%%%%%%%%%%%%%%%%%%%%%%%%%%%%%%%
\subsection{Basic algorithms}
%%%%%%%%%%%%%%%%%%%%%%%%%%%%%%%%%%%%%%%%%%%%%%%%%%%%%%%

Brute force is the most obvious idea for computing $\omega(n)$. All possible remainder \linebreak combinations according to proposition \ref{remainders}, statement (2), are processed  sequentially. To each of them, a fill\&cont procedure is applied as follows. For each of the residue classes, all positions of a previously empty array of sufficient length covered by that residue class are labelled. The first unlabelled position corresponds to $m+1$ where $m$ is the length of the related sequence. This length is then registered for searching the maximum possible length.

This na\"{i}ve algorithm is henceforth referred to as Basic Sequential Algorithm (BSA). It can be implemented as a recursive procedure as depicted in the pseudocode \ref{BSA}.\\

\begin{algorithm}[!h]
\caption{\ Basic Sequential Algorithm (BSA).} \label{BSA}
\begin{algorithmic} 
\Procedure{basic\_sequential}{arr,k}
   \For{i=1 to plist[k]-1}
      \State arr1=arr; fill\_array(arr1,i,plist[k])
      \If {k<n-1} basic\_sequential(arr1,k+1)
      \Else\ count\_array(arr1) \EndIf
   \EndFor 
\EndProcedure
\State arr=empty\_array	\Comment Sequence array
\State plist=[$p_{2}$,...,$p_n$]	\Comment Array of primes
\State k=1	\Comment Starting prime array index
\State basic\_sequential(arr,k)	\Comment Recursion
\end{algorithmic}
\end{algorithm}

The filling of a previously empty array of sufficient length from left to right \linebreak according to proposition \ref{permutation}, statement (3), is another simple idea for computing $\omega(n)$. All permutations $(\pi_2,\dots,\pi_n)$ of the given primes $\{p_2,\dots,p_n\}$ are processed  sequentially. Starting with position 1, the first prime $\pi_2$ is chosen and with it the residue class 1 mod $\pi_2$. After labelling all positions covered by that residue class, the next unlabelled position is searched and the next prime of the permutation under consideration is analogously used until all primes are consumed.

If in any case a prime divides the index of the next unlabelled position then this case can be omitted because only non-zero residue classes are appropriate. Handling all permutations where primes are not divisors of their related position-numbers will therefore discover all sequences of the maximum possible length. This Basic Permutation Algorithm is described in pseudocode \ref{BPA}.\\

\begin{algorithm}[!h]
\caption{\ Basic Permutation Algorithm (BPA).} \label{BPA}
\begin{algorithmic} 
\Procedure{basic\_permutation}{arr,plist,k,q}
   \For{i=k to n-1}
   \If {q $\not\equiv$ 0 (mod plist[i])}
      \State arr1=arr; fill\_array(arr1,q,plist[i])
      \If {k<n-1}
         \State plist1=plist; interchange(plist1[k],plist1[i])
         \State k1=k+1; q1=next\_free\_position(arr1)
         \State basic\_permutation(arr1,plist1,k1,q1)
      \Else\ count\_array(arr1) \EndIf
   \EndIf
   \EndFor 
\EndProcedure
\State arr=empty\_array	\Comment Sequence array
\State plist=[$p_{2}$,...,$p_n$]	\Comment Array of primes
\State k=1	\Comment Starting prime array index
\State q=1	\Comment Starting sequence position
\State basic\_permutation(arr,plist,k,q)		\Comment Recursion
\end{algorithmic}
\end{algorithm}
\bigskip

Both of the presented algorithms are not really effective for practical purpose.\linebreak A simple estimation of their complexities reveals the problem. The number $N_{BSA}$ of residue class combinations in the Basic Sequential Algorithm is
$$N_{BSA}=\prod_{i=2}^{n}{(p_i-1)}.$$
For the Basic Permutation Algorithm, we have
$$N_{BPA}\le (n-1)!\ll N_{BSA}.$$
However, $(n-1)!$ also grows too fast to be reasonable. In the upcoming sections, we describe possibilities of how to dramatically reduce the number of sequences which are needed to be considered for determining $\omega(n)$.\\

The Basic Permutation Algorithm as described above is ineffective for another \linebreak reason. There may exist different, equivalent permutations constituting the same \linebreak sequence. In this case, only one of these permutations must be examined. Selecting the smallest prime if there are any at choice, defines the additional rule of a Reduced Permutation Algorithm (RPA) \cite{Ziller_2005} which thereby also downsizes the number of permutations needed to be considered. This principle was yet applied in part $(2)\Rightarrow(3)$ of the proof of proposition \ref{permutation}. The Basic Permutation Algorithm \ref{BPA} indeed produces doublets of sequences.

\begin{prop} \label{equivalent}
Let $m,n\in\mN,\ n>1$, and $m=\omega(n)$.\\
Given a permutation $(\pi_2,\dots,\pi_n)$ of $\{p_2,\dots,p_n\}$ and a tuple $(q_2,\dots,q_n)$ so that the \linebreak sequence $\{1,\dots,m\}$ is completely covered by the residue classes $q_i$ mod $\pi_i$ when all $\pi_i$ were recursively assigned to the first free position $q_i$, respectively. And let $q_{k_1}\equiv q_{k_2}\ (mod\ \pi_{k_2})$, $k_1<k_2$, for any $k_1,k_2\in\{2,\dots,n\}$.
\begin{equation*} \begin{split}
&\forall\ q\in\{1,\dots,m\}\ \exists\ i\in\{2,\dots,n\}:q\equiv q_i\ (mod\ \pi_i)\\
&\quad where\ \pi_i\in\{p_2,\dots,p_n\},\ q_i\in\{1,\dots,m\},\ i=2,\dots,n,\\
&\quad \{\pi_2,\dots,\pi_n\}\setminus\{p_2,\dots,p_n\}=\emptyset,\\
&\quad q_i\not\equiv 0\ (mod\ \pi_i),\ i=2,\dots,n,\\
&\quad i<j \Rightarrow q_i<q_j,\ i,j\in\{2,\dots,n\},\\
&\quad q_2=1,\ and\\
&\quad q_i=min\{j\in\{1,\dots,m\}\mid \forall k<i : j\not\equiv q_k\ (mod\ \pi_k)\},\ i=3,\dots,n,\\
&and\\
&\exists\ k_1,k_2\in\{2,\dots,n\},\ k_1<k_2:q_{k_1}\equiv q_{k_2}\ (mod\ \pi_{k_2}).
\end{split} \end{equation*}
Then there exist an equivalent permutation $(\varrho_2,\dots,\varrho_n)$ of $\{p_2,\dots,p_n\}$ and a tuple $(r_2,\dots,r_n)$ according to proposition \ref{permutation}, statement (3) so that $\varrho_k=\pi_k$, $r_k=q_k$ for $k<k_1$ if exist, and $\varrho_{k_1}=\pi_{k_2},\ r_{k_1}=q_{k_1}$. Furthermore for $k>k_1$ if exist, $r_k\equiv q_j\ (mod\ \varrho_k)$ if $\varrho_k=\pi_j$.
\begin{equation*} \begin{split}
&\exists\ \varrho_i\in\{p_2,\dots,p_n\}\ \exists\ r_i\in\{1,\dots,m\},\ i=2,\dots,n\\
&\quad with\ \{\varrho_2,\dots,\varrho_n\}\setminus\{p_2,\dots,p_n\}=\emptyset,\\
&\quad r_i\not\equiv 0\ (mod\ \varrho_i),\ i=2,\dots,n,\\&\quad i<j \Rightarrow r_i<r_j,\ i,j\in\{2,\dots,n\},\\
&\quad r_2=1,\\
&\quad r_i=min\{j\in\{1,\dots,m\}\mid \forall\ k<i : j\not\equiv r_k\ (mod\ \varrho_k)\},\ i=3,\dots,n,\\
&\quad \forall\ q\in\{1,\dots,m\}\ \exists\ i\in\{2,\dots,n\}:q\equiv r_i\ (mod\ \varrho_i),\\
&so\ that\\
&\forall\ k\in\{2,\dots,n\},\ k<k_1: \varrho_k=\pi_k\land\ r_k=q_k,\\
&\varrho_{k_1}=\pi_{k_2},\ r_{k_1}=q_{k_1},\ and\\
&\forall\ k\in\{2,\dots,n\},\ k>k_1: r_k\equiv q_j\ (mod\ \varrho_k)\ if\ \varrho_k=\pi_j.
\end{split} \end{equation*}
\end{prop}

\begin{proof}
According to the assumptions, $r_i$ and $\varrho_i$ satisfy the requirements for $i\le k_1$. Given $r_k$ and $\varrho_k$ for $k_1\le k<i\le n$, we set
$$r_i=min\{j\in\{1,\dots,m\}\mid \forall\ k<i : j\not\equiv r_k\ (mod\ \varrho_k)\}.$$
There must exist one of the remaining primes $\pi_j\in\{\pi_{k_1},\dots,\pi_n\}\setminus\{\varrho_{k_1},\dots,\varrho_{i-1}\}$ with $r_i\equiv q_j\ (mod\ \pi_j)$ because $\{1,\dots,m\}$ is completely covered and $m=\omega(n)$ is maximum.\linebreak If there are several appropriate primes $\pi_j$ to choose from, we set  $\varrho_i=\pi_j$ for the \linebreak smallest $\pi_j$ possible. These $r_i$ and $\varrho_i$ inductively selected as described above complete the wanted permutation and tuple.
\end{proof}

\begin{exam}
We give an example for $n=8$. The primes to be considered are  3, 5, 7, 11, 13, 17, and 19. The result is $\omega(n)$=16.\\
The permutations\\
\hspace*{ 1cm}$\pi_2=3,\ \pi_3=13,\ \pi_4=11,\ \pi_5=7,\ \pi_6=5,\ \pi_7=17,\ \pi_8=19$, with\\
\hspace*{ 1cm}$q_2=1,\ q_3=2,\ q_4=3,\ q_5=5,\ q_6=6,\ q_7=8,\ q_8=9$,\\
and\\
\hspace*{ 1cm}$\varrho_2=5,\ \varrho_3=13,\ \varrho_4=11,\ \varrho_5=3,\ \varrho_6=7,
                            \ \varrho_7=17,\ \varrho_8=19$, with\\
\hspace*{ 1cm}$r_2=1,\ r_3=2,\ r_4=3,\ r_5=4,\ r_6=5,\ r_7=8,\ r_8=9$\\
are equivalent because\\
\hspace*{ 1cm}$k_1=2,\ k_2=6$ with $q_{k_1}\equiv q_{k_2}\ (mod\ \pi_{k_2})$, i.e. $1\equiv 6\ (mod\ 5)$, and\\
\hspace*{ 1cm}$\varrho_{k_1}=\pi_{k_2},\ r_{k_1}=q_{k_1}$.\\
The other required congruences are obvious when $\varrho_5=\pi_2$ and $\varrho_6=\pi_5$.
\end{exam}

The Reduced Permutation Algorithm (RPA) as depicted in pseudocode \ref{RPA} makes use of proposition \ref{equivalent} and skips all permutations for which there exists an equivalent permutation with a smaller prime at a compatible position.\\

\begin{algorithm}[!h]
\caption{\ Reduced Permutation Algorithm (RPA).} \label{RPA}
\begin{algorithmic} 
\Procedure{reduced\_permutation}{arr,plist,k,qlist}
   \For{i=k to n-1}
   \If {qlist[k] $\not\equiv$ 0 (mod plist[i])}
   \If {forall j<k : qlist[j]$\not\equiv$qlist[k] (mod plist[k]) or plist[j]<plist[k]}
      \State arr1=arr; fill\_array(arr1,qlist[k],plist[i])
      \If {k<n-1}
         \State plist1=plist; interchange(plist1[k],plist1[i])
         \State qlist[k+1]=next\_free\_position(arr1)
         \State reduced\_permutation(arr1,plist1,k+1,qlist)
      \Else\ count\_array(arr1) \EndIf
   \EndIf
   \EndIf
   \EndFor 
\EndProcedure
\State arr=empty\_array	\Comment Sequence array
\State plist=[$p_{2}$,...,$p_n$]	\Comment Array of primes
\State k=1	\Comment Starting prime array index
\State qlist=empty\_array	\Comment List of first unlabelled positions
\State qlist[k]=1	\Comment Starting sequence position
\State reduced\_permutation(arr,plist,k,qlist)		\Comment Recursion
\end{algorithmic}
\end{algorithm}

%%%%%%%%%%%%%%%%%%%%%%%%%%%%%%%%%%%%%%%%%%%%%%%%%%%%%%%
\subsection{Linear programming}
%%%%%%%%%%%%%%%%%%%%%%%%%%%%%%%%%%%%%%%%%%%%%%%%%%%%%%%

There is another, direct way of computing $\omega(n)$. The problem of determining the maximum length of a sequence covered by a choice of residue classes according to proposition \ref{remainders}, statement (2), can be reformulated as a linear program. This idea was recently applied by Resta to a similar topic \cite{Resta_2015}. The possibility of simply using standard software seems to be tempting. It is not an algorithmic idea itself. But we outline it in this section for the sake of completeness, simplicity, and mathematical straightforwardness.\\

Resuming statement (2) of proposition \ref{remainders}, there exists one non-zero residue class for every prime $p_2,\dots,p_n$ so that every integer of the sequence $\{1,\dots,m\}$ belongs to one of them.
\begin{equation*} \begin{split}
&\exists\ a_i\in\{1,\dots,p_i-1\},\ i=2,\dots,n:\\
&\forall\ q\in\{1,\dots,m\}\ \exists\ i\in\{2,\dots,n\}:q\equiv a_i\ (mod\ p_i).
\end{split} \end{equation*}

For every possible residue class, we define a binary variable \mbox{$x_{i,j}\in\{0,1\}$},\linebreak
 $i\in\{2,\dots,n\}$, $j\in\{1,...,p_i-1\}$ where $x_{i,j}=1$ if and only if $j=a_i$.

For the description of the sequence characterised as above, we formulate two sets of constraints. First, exactly one remainder should be chosen for each prime. Second, every position in the sequence $\{1,\dots,m\}$ should be covered by at least one of the residue classes.
\begin{equation} \begin{aligned}
&\sum_{j=1}^{p_i-1}{x_{i,j}}&&=1,\quad i=2,\dots,n,\\
&\sum_{\substack{i=2\\p_i\ \ndiv\ q}}^{n}{x_{i,q\ mod\ p_i}}&&\ge 1,\quad q=1,...,m.
\end{aligned} \end{equation}

The question of interest is whether or not there exist feasible solutions of these constraints. An objective function is not really needed. So, a dummy objective like
$$\max\sum_{i=2}^{n}{ \sum_{j=1}^{p_i-1}{x_{i,j}} }$$
completes the integer linear program (ILP) with binary variables.

In order to calculate $\omega(n)$, a trial and error approach like nested intervals or the like might be applied, solving the described ILP for different $m$. If the system has any feasible solutions for $m$ and none for $m+1$ then $\omega(n)=m$.\\

The search for the maximum suitable $m$ can also be embedded into a single,\linebreak extended ILP. Let $m_1<\omega(n)$ and $m_2>\omega(n)$ be tentative assumptions. We define additional binary variables $y_k\in\{0,1\},\ k\in\{m_1,\dots,m_2\}$ where $y_k=1$ if and only if position $k$ in the sequence is covered by any of the residue classes under consideration, and formulate the following linear program with binary variables.\\

\begin{equation} \begin{aligned} \label{ILP2}
&\qquad\qquad\qquad\max&&\sum_{k={m_1}}^{m_2}{2^{m_2-k}\mdot y_k},\\
&\sum_{j=1}^{p_i-1}{x_{i,j}}&&=1,\quad i=2,\dots,n,\\
&\sum_{\substack{i=2\\p_i\ \ndiv\ q}}^{n}{x_{i,q\ mod\ p_i}}&&\ge 1,\quad q=1,...,m_1-1,\\
&\sum_{\substack{i=2\\p_i\ \ndiv\ k}}^{n}{x_{i,k\ mod\ p_i}}\quad- y_k&&\ge 0,\quad k=m_1,...,m_2.
\end{aligned} \end{equation}

This linear program can be solved by using an established ILP solver. There are three kinds of solutions:

  \begin{enumerate}
\item[(1)]
$y_k=0$ for all $k=m_1,\dots,m_2$.

The choice of $m_1$ was too large, $\omega(n)<m_1$. Another try with a smaller $m_1$ is needed.
\item[(2)]
$y_k=1$ for all $k=m_1,\dots,m_2$.

The choice of $m_2$ was too small, $\omega(n)\ge m_2$. Another try with a larger $m_2$ is needed.
\item[(3)]
$y_k=1$ for all $k=m_1,\dots,m<m_2$,\quad $y_{m+1}=0$.\\
Then $\omega(n)=m$. An appropriate sequence of the length $m$ can be derived from the solution for the variables $x_{i,j}$. A longer sequence with the required properties cannot exist because of the choice of the objective function. Each of its coefficients is larger than the sum of the following ones.
  \end{enumerate}

The above described way of calculating $\omega(n)$ is very simple but its feasibility is limited as well. The required computation time rapidly grows with the increasing dimension of the problem.

%%%%%%%%%%%%%%%%%%%%%%%%%%%%%%%%%%%%%%%%%%%%%%%%%%%%%%%
\subsection{Bounding the remaining number of coprimes}
%%%%%%%%%%%%%%%%%%%%%%%%%%%%%%%%%%%%%%%%%%%%%%%%%%%%%%%

The principles of the basic algorithms described above can only be reasonably utilised if the exclusion of many of the theoretically possible cases from the calculation can be conclusively substantiated. The Basic Sequential Algorithm processes all potential combinations of residue classes. When particular combinations can be excluded from the analysis, the computational effort can be improved upon as is done in the Reduced Permutation Algorithm. A remainder combination may be discarded because it actually cannot cover a sequence of the considered length. This decision must be based on proved criteria, which we derive in the following section.\\

The general idea was independently realised by Hagedorn \cite{Hagedorn_2009} and Morack \cite{Morack_2014} in very similar approaches. Euler's totient function $\varphi(n)$ is known to be the number of positive integers up to a given integer $n$ that are coprime to $n$. Costello and Watts used a specific generalisation of this function for the construction of a computational bound of Jacobsthal's function \cite{Costello_Watts_2012}. This generalisation counts the number of integers in a given sequence which are coprime to given primes. Following this basic idea, we are looking at the topic the other way round.

The function $\psi(a,m,k)$ is the number of integers in the sequence $a+1,\dots,a+m$ which are not coprime to $P_k=\prod_{i=2}^k p_i$, i.e. which are divisible by one of the primes $p_2,\dots,p_k$. We emphasize that the first prime $p_1=2$ was omitted from the considerations in this context. In other words, $\psi(a,m,k)$ is the number of covered positions in the given sequence of length $m$ where the set of remainders is condensed in $a$.

\begin{defi}
For $a\in\mZ$, $m\in\mN$, $k\in\mN\ge2$, we define
\begin{equation*} \begin{split}
\psi(a,m,1)&=0,\\
\psi(a,m,k)&=|\{q\in\{1,\dots,m\}\mid \exists\ i\in\{2,\dots,k\}:a+q\equiv 0\ (mod\ p_i)\}|,\\
\nu(a,m,1)&=0,\ and\\
\nu(a,m,k)&=\psi(a,m,k)-\psi(a,m,k-1).
\end{split} \end{equation*}
\end{defi}

The difference $\nu(a,m,k)$ corresponds to the number of integers in the sequence which are divisible by $p_k$ but coprime to the primes $p_2,\dots,p_{k-1}$, i.e. the newly covered positions.

\begin{rema}
The sequence is completely covered if $\psi(a,m,n)=m$. The successive contribution of prime $p_k$ is expressed by $\nu(a,m,k)$. Thus for $2\le k\le n$, we get
$$\omega(n)=max\{m\in\mN\mid \exists\ a\in\mZ : \psi(a,m,n)=\psi(a,m,k-1)+\sum_{i=k}^n \nu(a,m,i)=m\}.$$
\end{rema}

Let $a$ solve the simultaneous congruences $a\equiv -a_j\ (mod\ p_j)$ for $j=2,\dots,n$. Given a tentative length $m$ of the sequence, the further processing of the corresponding set of remainders $\{a_j\}$ can be skipped if
$$\sum_{i=k}^n \nu(a,m,i)<m-\psi(a,m,k-1)$$
can be proved for any $k$. Whereas $\psi(a,m,k-1)$ can directly be counted, we below derive suitable upper bounds for $\sum_{i=k}^n \nu(a,m,i)$ for an early and effective decision whether this inequality applies.\\

In a first step, we simplify our estimate and make it independent on specific\linebreak sequences, i.e independent on $a$, or $a_j$, respectively. We consider the lowest possible number $\psi_{min}(m,k)$ of $m$ consecutive integers which are divisible by one of the primes $p_2,\dots,p_k$, and the highest possible number $\nu_{max}(m,k)$ of $m$ consecutive integers which are divisible by $p_k$ but coprime to $p_2,\dots,p_{k-1}$.

\begin{defi} \label{psi_min}
Let $m,k\in\mN$. Then
$$\psi_{min}(m,k)=\min_{a\in\mZ} \psi(a,m,k),$$
$$\nu_{max}(m,k)=\max_{a\in\mZ} \nu(a,m,k).$$
\end{defi}

The term $\nu_{max}(m,k)$ can be bounded as the result of combinatorial considerations as follows.

\begin{lemm} \label{bound}
Let $\lfloor x\rfloor$ denote the maximum integer not exceeding $x\in\mR$, $m\in\mN$, and\linebreak $k\in\mN\ge2$. Then
$$\nu_{max}(m,k)\le r_{m,k} - \psi_{min}(r_{m,k},k-1)$$
where $r_{m,k}=1+\left\lfloor\frac{m-1}{p_k}\right\rfloor$.
\end{lemm}

\begin{proof}
The number of multiples of $p_k$ in a sequence of $m$ consecutive integers is at most $r_{m,k}$. Every sequence contains at least $r_{m,k}-1$ such terms. Let $p_k\mdot (a+j)$, $a\in\mZ$ and $j=1,\dots,r_{m,k}$, be the representation of $r_{m,k}$ consecutive multiples of $p_k$.

Every term $p_k\mdot (a+j)$ is divisible by $p_i$ for  $i\in\{2,\dots,k-1\}$ if and only if $p_i\mdiv (a+j)$ because $p_i\perp p_k$.
$$p_i\mdiv (a+j)\Leftrightarrow p_i\mdiv p_k\mdot (a+j).$$
The sequence $\{a+1,\dots,a+r_{m,k}\}$ contains multiples of $p_i$ at the same positions $j$\linebreak as the arithmetic progression $\{p_k\mdot (a+j)\}$ does. The number of elements in this\linebreak progression which are divisible by any $p_i,\ i=2,\dots,k-1$ is therefore\linebreak $\psi(a,r_{m,k},k-1)\ge\psi_{min}(r_{m,k},k-1)$.
\end{proof}

The term $r_{m,k}$ can directly be calculated whereas the function $\psi_{min}(m,k)$ requires complex computations. For this purpose, we applied a brute force algorithm very\linebreak similar to the Basic Sequential Algorithm \ref{BSA}. While BSA intends to maximise the\linebreak number of covered positions in a given array, the computation of $\psi_{min}(m,k)$ needs to minimise it. We provide a table of function values of $\psi_{min}(m,k)$ for $m\le500$ and $k\le11$ in an ancillary file.

The function $\psi_{min}(m,k)$ grows very slowly with increasing $k$ while the time needed for it explodes. So, it must be sufficient to limit the maximum considered prime for practical application.

\begin{coro}
Let $a\in\mZ$, $m,n,t\in\mN$, $k\in\mN\ge2$, and $r_{m,i}=1+\left\lfloor\frac{m-1}{p_i}\right\rfloor$, $i=k,\dots,n$.
For any $t<k$,
$$\sum_{i=k}^n \nu(a,m,i) \le \sum_{i=k}^n r_{m,i} - \sum_{i=k}^n \psi_{min}(r_{m,i},t).$$
\end{coro}

\begin{proof}
According to lemma \ref{bound}, $\nu(a,m,i) \le \nu_{max}(m,i)\le r_{m,i} - \psi_{min}(r_{m,i},i-1)$ holds for any $i=k,\dots,n$. The function $\psi_{min}(m,k)$ is monotonically increasing in $k$\linebreak because the contribution of the respective next greater prime cannot be negative.\linebreak So, $\psi_{min}(r_{m,i},i-1) \ge \psi_{min}(r_{m,i},t)$ is a consequence of $t<k\le i$.
\end{proof}

With the help of this corollary, we can efficiently improve the na\"{i}ve sequential algorithm. We recall that the further processing of the set of remainders $\{a_j\}$ for a tentative sequence length $m$ can be skipped if
$\sum_{i=k}^n \nu(a,m,i)<m-\psi(a,m,k-1)$ can be proved for any $k$. In this context, it is sufficient to verify the criterion
\begin{equation}\label{crit}
\sum_{i=k}^n r_{m,i} - \sum_{i=k}^n \psi_{min}(r_{m,i},t)<m-\psi(a,m,k-1)
\end{equation}
for a suitable $t<k$.

\ 

This is the fundamental idea of the Discarding Sequential Algorithm (DSA) \ref{DSA} to compute values of $\omega(n)$. Starting with a sufficiently long empty array and a tentative sequence length $m$, the parameter $m$ is increased if a longer sequence was found.

It has shown to be more efficient if the first cycles, which employ only small primes, were processed without checking for the possibility of rejection. In later stages of the recursion, say from $p_{k^*}$ on, the criterion can be applied. If the table of $\psi_{min}$ includes the corresponding values for $t=k-1$ then they are preferred. Otherwise, the largest possible $t$ must be used.

Another modification can slightly speed up the algorithm. The criterion \ref{crit} need not be compared for the entire length $m$ of the considered sequence. It is sufficient to take the maximum length $m^*$ of a subsequence into account which includes all\linebreak uncovered positions. While counting $\psi(a,m,k-1)$, the appropriate $m^*$ can simply be determined as well.\\

\begin{algorithm}[!h]
\caption{\ Discarding Sequential Algorithm (DSA).} \label{DSA}
\begin{algorithmic}%[4] 
\Procedure{discarding\_sequential}{arr,k}
   \For{i=1 to plist[k]-1}
      \State arr1=arr; fill\_array(arr1,i,plist[k])
      \If {k<n-1}
         \State go\_on=true
         \If { k$\ge\text{k}^*$}
            \If { criterion \ref{crit} fulfilled} go\_on=false \EndIf
         \EndIf
         \If {go\_on=true} discarding\_sequential(arr1,k+1)
         \EndIf
      \Else\ count\_array(arr1)
         \If {longer\_sequence\_found} increase\_m \EndIf
      \EndIf
   \EndFor 
\EndProcedure
\State arr=empty\_array	\Comment Sequence array
\State m=starting\_sequence\_length	\Comment Starting sequence length
\State plist=[$p_2$,...,$p_n$]	\Comment Array of primes
\State k=1	\Comment Starting prime array index
\State $\text{k}^*$=starting\_index\_for\_criterion	\Comment Starting prime array index for criterion
\State discarding\_sequential(arr,k)	\Comment Recursion
\end{algorithmic}
\end{algorithm}

\pagebreak\ \\

Although the DSA algorithm includes an efficient way to recognise inappropriate remainder combinations, it is time-consuming especially for large primes. The number of permutations to be considered in the RPA algorithm, however  is much lower than the number of possible remainder combinations. The connection of these two principles in a Combined Reduced Permutation and Discarding Sequential Algorithm (CRPDSA) \ref{CRPDSA} can make the advantages of both ideas work together \cite{Ziller_Morack_2015}.

For smaller primes, the DSA algorithm is applied sequentially up to a fixed prime. The remaining primes are handled with the RPA algorithm which was extended by a test of the DSA-criterion. The number of necessary cases can thus rigorously be reduced. An optimum point for the switch between both sub-algorithms was empirically found at  $p_n/3$ when $p_n$ is the largest prime considered. We request $p_{k^*}$<$p_n/3$.\\

\pagebreak

\begin{algorithm}[!h]
\caption{\ Combined Reduced Permutation and \newline \hspace*{2.45cm}\ Discarding Sequential Algorithm (CRPDSA).} \label{CRPDSA}

\begin{algorithmic} 
\Procedure{combined\_discarding}{arr,plist,k,qlist}
 \If {$p_k$<$p_k/3$}	\Comment DSA part

   \For{i=1 to plist[k]-1}
      \State arr1=arr; fill\_array(arr1,i,plist[k]); go\_on=true
      \If { k$\ge\text{k}^*$}
         \If { criterion \ref{crit} fulfilled} go\_on=false \EndIf
      \EndIf
      \If {go\_on=true} combined\_discarding(arr1,plist,k+1,qlist)
      \EndIf
   \EndFor 

 \Else
   \If {empty\_qlist} fill\_qlist \EndIf

   \For{i=k to n-1}	\Comment RPA part
   \If {qlist[k] $\not\equiv$ 0 (mod plist[i])}
   \If {forall j<k : qlist[j]$\not\equiv$qlist[k] (mod plist[k]) or plist[j]<plist[k]}
      \State arr1=arr; fill\_array(arr1,qlist[k],plist[i])
      \If {k<n-1} go\_on=true
         \If { criterion \ref{crit} fulfilled} go\_on=false \EndIf
         \If {go\_on=true}
           \State plist1=plist; interchange(plist1[k],plist1[i])
           \State qlist[k+1]=next\_free\_position(arr1)
           \State combined\_discarding(arr1,plist1,k+1,qlist)
           \EndIf
      \Else\ count\_array(arr1)
         \If {longer\_sequence\_found} increase\_m \EndIf
      \EndIf
   \EndIf
   \EndIf
   \EndFor 

\EndIf
\EndProcedure
\State arr=empty\_array	\Comment Sequence array
\State m=starting\_sequence\_length	\Comment Starting sequence length
\State plist=[$p_{2}$,...,$p_n$]	\Comment Array of primes
\State qlist=empty\_array	\Comment List of first unlabelled positions
\State k=1	\Comment Starting prime array index
\State $\text{k}^*$=starting\_index\_for\_criterion	\Comment Starting prime array index for criterion
\State combined\_discarding(arr,plist,k,qlist)	\Comment Recursion
\end{algorithmic}

\end{algorithm}

%%%%%%%%%%%%%%%%%%%%%%%%%%%%%%%%%%%%%%%%%%%%%%%%%%%%%%%
\subsection{Counting the actually remaining number of coprimes}
%%%%%%%%%%%%%%%%%%%%%%%%%%%%%%%%%%%%%%%%%%%%%%%%%%%%%%%

The last suggestion of this paper, how to compute $\omega(n)$, is an improvement of the DSA algorithm \ref{DSA}. The actually remaining number of coprimes in the sequence is exactly counted for every pending prime
instead of using a general, estimated bound for it. This is more time-consuming for a specific sequence under consideration. On the other hand, exact counts are often lower than the bounds. So, corresponding remainder constellations can be rejected much earlier. This compensates for the time effort spent for counting the exact frequencies.

We redefine some functions of the last subsection in a more general way. They are not restricted to require primes in their ascending order any more. The idea of permutations plays a role again.

\begin{defi}
Let $(\pi_2,\dots,\pi_n)$ be an arbitrary but fixed permutation of $\{p_2,\dots,p_n\}$.
For $a\in\mZ$, $m\in\mN$, $k\in\mN\ge2$, we define
\begin{equation*} \begin{split}
\psi(a,m,1)&=0,\\
\psi(a,m,k)&=|\{q\in\{1,\dots,m\}\mid \exists\ i\in\{2,\dots,k\}:a+q\equiv 0\ (mod\ \pi_i)\}|,\\
\nu(a,m,1)&=0,\ and\\
\nu(a,m,k)&=\psi(a,m,k)-\psi(a,m,k-1).
\end{split} \end{equation*}
\end{defi}

The difference $\nu(a,m,k)$ now corresponds to the number of integers divisible by $\pi_k$ but coprime to $\pi_2,\dots,\pi_{k-1}$. Again, let $m$ be the tentative length of the sequence, and $\{a_j\}$ a set of remainders where $a$ is the solution of the simultaneous congruences $a\equiv -a_j\ (mod\ \pi_j)$, $j=2,\dots,n$. The further processing of the corresponding set of remainders $\{a_j\}$ can then be skipped if
$$\sum_{i=k}^n \nu(a,m,i)<m-\psi(a,m,k-1)$$
can be proved for any $k$.

Every position in the sequence corresponds to a definite residue class for each prime. So, the number of remaining coprimes which belong to a residue class\linebreak can easily be counted. By this means, we derive individual bounds for $\nu(a,m,i)$\linebreak depending on the specific choice of $\{a_j\}$.

\begin{defi}
Let $(\pi_2,\dots,\pi_n)$ be an arbitrary but fixed permutation of $\{p_2,\dots,p_n\}$,\\
$m,n,k,t\in\mN$ with $2\le k<t\le n$, and $a\in\mZ$ with $a\equiv -a_j\ (mod\ \pi_j)$, $j=2,\dots,k$.\linebreak
Then we define
\begin{equation*} \begin{split}
S_1&=\{1,\dots,m\},\\
S_k&=\{q\in S_{k-1}\mid a+q\not\equiv 0\ (mod\ \pi_k)\},\\
\text{and for}\ r_t\in\{1,\dots,\pi_t-1\}\qquad&\hspace{15cm}\\
\varrho(a,m,k,t,r_t)&=|\{q\in S_{k}\mid q\equiv r_t\ (mod\ \pi_t)\}|,\ and\\
\varrho_{max}(a,m,k,t)&=\max_{r_t\in\{1,\dots,\pi_k-1\}} \varrho(a,m,k,t,r_t).
\end{split} \end{equation*}
\end{defi}

The set of positions of a sequence $a+1,\dots,a+m$ which are coprime to all primes up to $\pi_k$ is denoted by $S_k$, i.e. the set of so far uncovered positions. Given the set $S_k$, \ $\varrho(a,m,k,k+1,r_{k+1})$ is the number of positions of the sequence which can be covered in the next step if $a_{k+1}=r_{k+1}$ was chosen. We now relate the function $\nu$ to $\varrho_{max}$.

\begin{lemm} \label{count}
Let $a\in\mZ$, and $m,n,k,t\in\mN$ with $2\le k<t\le n$. Then
$$\nu(a,m,t)\le\varrho_{max}(a,m,k,t).$$
\end{lemm}

\begin{proof}
By definition, there exists an $r_t\in\{1,\dots,\pi_t-1\}$ so that\\
$\nu(a,m,t)=\varrho(a,m,t-1,t,r_t)$.\qquad For $k<t$, we get\\
$\varrho(a,m,t-1,t,r_t)\le\varrho(a,m,k,t,r_t)\le\varrho_{max}(a,m,k,t)$ 
because $S_k\supseteq S_{t-1}$.
\end{proof}

In consequence of lemma \ref{count}, it follows
$$\sum_{i=k}^n \nu(a,m,i)\le\sum_{i=k}^n\varrho_{max}(a,m,k-1,i),$$
because $i\ge k$ and $k<t$, and we can now formulate the criterion
\begin{equation} \label{crit2}
\sum_{i=k}^n\varrho_{max}(a,m,k-1,i)<m-\psi(a,m,k-1)
\end{equation}
for the rejection of the further processing of the corresponding set of remainders $\{a_j\}$ for a tentative sequence length $m$.\\

The key idea of the effective realisation of this criterion in an algorithm is\linebreak choosing the residue class with the maximum possible number of newly covered\linebreak positions $\max_{t\in\{k,\dots,n\}}\varrho_{max}(a,m,k-1,t)$ in every recursion cycle of level $k$. This choice implies a permutation algorithm because it selects $\pi_t$ as the next prime. At the same time, it selects the residue class $r_t\ mod\ \pi_t$ with\\
$\varrho(a,m,k-1,t,r_t)=\max_{t\in\{k,\dots,n\}}\varrho_{max}(a,m,k-1,t)$. Thus, the so-called Greedy\linebreak Permutation Algorithm (GPA) \ref{GPA} combines permutations with respect to primes and to remainders.

After level $k+1$ of recursion was done for that specific choice, the following\linebreak cycles of level $k$ do not need to consider the consumed $r_t\ mod\ \pi_t$ again because all\linebreak combinations of it were thereby examined. Otherwise, equivalent permutations would be considered. For the following cycles of level $k$, exclusively, the corresponding $\varrho(a,m,k-1,t,r_t)$ is set to be 0, and $\varrho_{max}(a,m,k-1,t)$ is determined anew.

Because of the specific choice of the remainders $r_t$, the Greedy Permutation \linebreak Algorithm (GPA) is confined to process only prime permutations with \linebreak $\varrho(a,m,i-1,i,r_i)\le\varrho(a,m,j-1,j,r_j)$ for $i<j$, i.e. the contribution of the primes is monotonically decreasing with the permutation order. Furthermore, a cyclic\linebreak inner permutation of primes with the same contribution is prevented by blocking those $r_t\ mod\ \pi_t$ always consumed at an earlier level.\\

\pagebreak

\begin{algorithm}[!h]
\caption{\ Greedy Permutation Algorithm (GPA).} \label{GPA}
\begin{algorithmic} 
\Procedure{greedy\_permutation}{arr,k,ftab}
\If {k<$\text{k}^*$}	\Comment Sequential part
   \For{i=1 to plist[k]-1}
      \State arr1=arr; fill\_array(arr1,i,plist[k])
      \If {k=$\text{k}^*$-1}
         \State  ftab1=ftab
         \State fill\_frequency\_table\_of\_remainders(ftab1)
      \EndIf
      \State greedy\_permutation(arr1,k+1,ftab1)
   \EndFor 
\Else
   \If {k<n-1}	\Comment Greedy permutation part
      \State update\_$\varrho_{max}(\dots,t)$
      \State go\_on=true
      \If { criterion \ref{crit2} fulfilled} go\_on=false \EndIf

       \If {go\_on=true}	\Comment Permutation level k+1
          \State select\_appropriate\_$r_t$\_and\_$\pi_t$
          \State arr1=arr; fill\_array(arr1,$r_t$,$\pi_t$)
          \State  ftab1=ftab
          \State update\_frequency\_table\_of\_remainders(ftab1)
          \State greedy\_permutation(arr1,k+1,ftab1)
      \EndIf

       \State delete\_frequency\_of\_$r_t$\_mod\_$\pi_t$(ftab)
       \If {exists\_non-zero\_frequency\_mod\_$\pi_t$(ftab)}	\Comment Permutation level k
          \State greedy\_permutation(arr,k,ftab)
      \EndIf
   \Else\ count\_array(arr)
         \If {longer\_sequence\_found} increase\_m \EndIf
   \EndIf
\EndIf
\EndProcedure

\State arr=empty\_array	\Comment Sequence array
\State m=starting\_sequence\_length	\Comment Starting sequence length
\State plist=[$p_{2}$,...,$p_n$]	\Comment Array of primes
\State k=1	\Comment Starting prime array index
\State $\text{k}^*$=starting\_index\_for\_criterion	\Comment Starting prime array index for criterion
\State ftab=empty\_table	\Comment Frequency table of remainders

\State greedy\_permutation(arr,k,ftab)	\Comment Recursion
\end{algorithmic}
\end{algorithm}

\pagebreak

It can be  shown that the GPA algorithm is again more efficient if the first cycles, which employ only small primes, were processed simply sequentially without any checking for the possibility of rejection. This corresponds to the Basic Sequential\linebreak Algorithm \ref{BSA}. In later stages of the recursion, say from $p_{k^*}$ on, the new criterion will be applied, and the recursion changes to the true Greedy Permutation Algorithm.\\

%%%%%%%%%%%%%%%%%%%%%%%%%%%%%%%%%%%%%%%%%%%%%%%%%%%%%%%
\subsection{Parallel processing}
%%%%%%%%%%%%%%%%%%%%%%%%%%%%%%%%%%%%%%%%%%%%%%%%%%%%%%%

All depicted explicit algorithms can be processed in parallel in a simple way. All of them are organised as treelike recursive procedures. Given a fixed recursion level $k^*$, the recurrent computation can be prepared in a first step by executing the recursive procedures until level $k^*$, only. All relevant parameters defining the next procedure call for level  $k^*+1$ are written to a separate parameter file. Now in a second step, the subsequent subtrees or series of them, each starting at level  $k^*+1$, can be processed to their end on different cores, processors, or even on different computers. The last step summarises all individual results of these subprocesses, including all sequences of maximum length.

Parallel computation can be used most efficiently by a prior sorting of the subtrees of level $k^*+1$ by descending $\psi(a,m,k)$. This  leads to an earlier recognition of longer sequences on average. The rapid enlargement of the respective array prevents from examining many too small sequences unnecessarily, and therefore speeds up the entire approach.\\

\pagebreak

%%%%%%%%%%%%%%%%%%%%%%%%%%%%%%%%%%%%%%%%%%%%%%%%%%%%%%%
%%%%%%%%%%%%%%%%%%%%%%%%%%%%%%%%%%%%%%%%%%%%%%%%%%%%%%%
\section{Results}
%%%%%%%%%%%%%%%%%%%%%%%%%%%%%%%%%%%%%%%%%%%%%%%%%%%%%%%
%%%%%%%%%%%%%%%%%%%%%%%%%%%%%%%%%%%%%%%%%%%%%%%%%%%%%%%

The algorithms described in section \ref{Algorithms} act very differently concerning its computational effort. We computed the entire values of the function $h(n)$ where all involved primes could be represented as single byte, i.e. $p_n\le251$. In these calculations, we used\linebreak several algorithms each suitable regarding computation time. The equality of corresponding results served as an implicit verification of the correctness of the implementations. Our data are also in accord with all published data for $p_n\le227$ \cite{Hagedorn_2009}.

For every investigated $n$, we performed an exhaustive retrieval for all existing\linebreak sequences of maximum length while searching for $\omega(n)$. As far as we know, this has not been done ever before. We provide this data in ancillary files.

%%%%%%%%%%%%%%%%%%%%%%%%%%%%%%%%%%%%%%%%%%%%%%%%%%%%%%%
\subsection{Calculated data}
%%%%%%%%%%%%%%%%%%%%%%%%%%%%%%%%%%%%%%%%%%%%%%%%%%%%%%%

All algorithms output values of the function $\omega(n)$ according to definition \ref{omega} as their main result. In addition, every compatible sequence of the appropriate length $\omega(n)$ was recorded. The values of $h(n)$, see definition \ref{h}, were deduced from $\omega(n)$ by\linebreak applying corollary \ref{h_omega}.\\

An initial illustration of the results is outlined in figure 1. We graph the values\linebreak of $h(n)$ and the number of sequences of length $\omega(n)$. This number considerably varies in a non-obvious manner.\\

\begin{figure}[!h]
  \centering
  \includegraphics[width=0.95\textwidth]{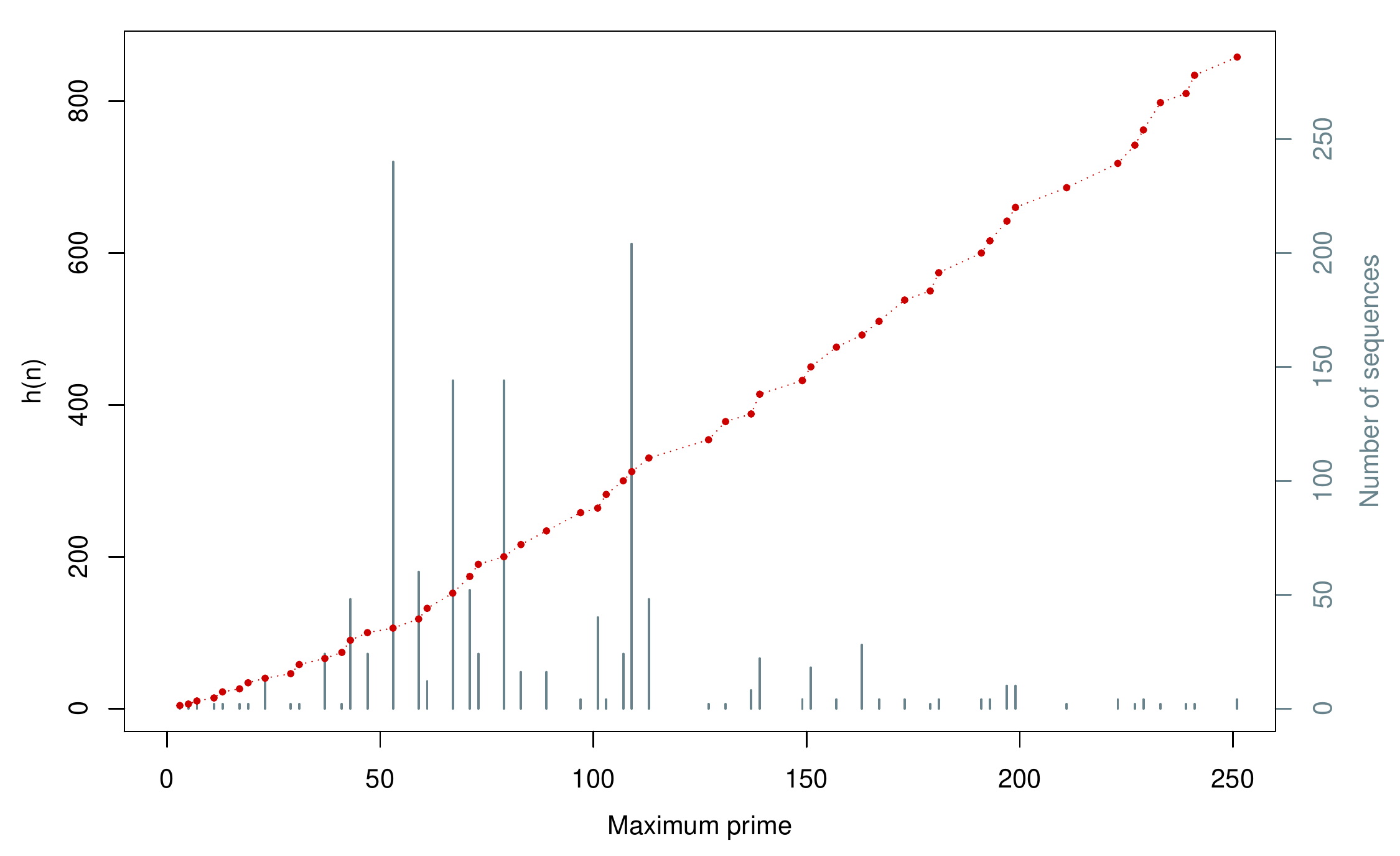}
%  \caption{Values of the function $h(n)$ and the number of respective sequences of maximum length.}
\end{figure}
\stepcounter{figure}
\ \ Figure 1: Values of the function $h(n)$ and the number of respective sequences

\ \ of maximum length.\\

The complete data are provided in table 1 including the maximum processed prime $p_n$, the values of $h(n)$ and $\omega(n)$, and the number of sequences $n_{seq}$ for every $n\le54$.\\

\begin{table}[!h]
  \centering
\begin{tabular}{cc}	
\begin{tabular}{!\vl rr|rrr !\vl}
  \noalign{\hrule height 2pt} 
\rule{0pt}{14pt}$n$&$p_n$ &$h(n)$& $\omega(n)$& $n_{seq}$\\[2pt]
  \noalign{\hrule height 2pt} 
  \rule{0pt}{14pt}1&2&2&$-$&$-$\\
2&3&4&1&1\\
3&5&6&2&2\\
4&7&10&4&2\\
5&11&14&6&2\\
6&13&22&10&2\\
7&17&26&12&2\\
8&19&34&16&2\\
9&23&40&19&12\\
10&29&46&22&2\\
11&31&58&28&2\\
12&37&66&32&24\\
13&41&74&36&2\\
14&43&90&44&48\\
15&47&100&49&24\\
16&53&106&52&240\\
17&59&118&58&60\\
18&61&132&65&12\\
19&67&152&75&144\\
20&71&174&86&52\\
21&73&190&94&24\\
22&79&200&99&144\\
23&83&216&107&16\\
24&89&234&116&16\\
25&97&258&128&4\\
26&101&264&131&40\\
27&103&282&140&4\\
  [2pt]
  \noalign{\hrule height 2pt}
\end{tabular}
&
\begin{tabular}{!\vl rr|rrr !\vl}
  \noalign{\hrule height 2pt} 
\rule{0pt}{14pt}$n$&$p_n$ &$h(n)$& $\omega(n)$& $n_{seq}$\\[2pt]
  \noalign{\hrule height 2pt} 
  \rule{0pt}{14pt}28&107&300&149&24\\
29&109&312&155&204\\
30&113&330&164&48\\
31&127&354&176&2\\
32&131&378&188&2\\
33&137&388&193&8\\
34&139&414&206&22\\
35&149&432&215&4\\
36&151&450&224&18\\
37&157&476&237&4\\
38&163&492&245&28\\
39&167&510&254&4\\
40&173&538&268&4\\
41&179&550&274&2\\
42&181&574&286&4\\
43&191&600&299&4\\
44&193&616&307&4\\
45&197&642&320&10\\
46&199&660&329&10\\
47&211&686&342&2\\
48&223&718&358&4\\
49&227&742&370&2\\
50&229&762&380&4\\
51&233&798&398&2\\
52&239&810&404&2\\
53&241&834&416&2\\
54&251&858&428&4\\
  [2pt]
  \noalign{\hrule height 2pt}
\end{tabular}
\end{tabular}
 \caption{Computation results.}
\end{table}

%%%%%%%%%%%%%%%%%%%%%%%%%%%%%%%%%%%%%%%%%%%%%%%%%%%%%%%
\subsection{Computation time}
%%%%%%%%%%%%%%%%%%%%%%%%%%%%%%%%%%%%%%%%%%%%%%%%%%%%%%%

All algorithms except ILP were implemented and executed on a common PC with an i7 processor at boosted 3.9 GHz in a single thread application. For solving the\linebreak respective integer linear programs according to the equations \ref{ILP2}, however, we\linebreak utilised SYMPHONY software \cite{SYMPHONY,COIN} which used all threads in parallel. The corresponding time needs were enlarged by an empirically estimated factor to make them comparable.\\

The time consumption of each algorithm rapidly grows with the number of primes. Figure 2 depicts this growth as well as the variation between different algorithms.\linebreak A quasi-logarithmic time scale $log(1+t)$ was applied where $t$ was the exact time need rounded to full seconds.\\

\begin{figure}[!h]
  \centering
  \includegraphics[width=0.95\textwidth]{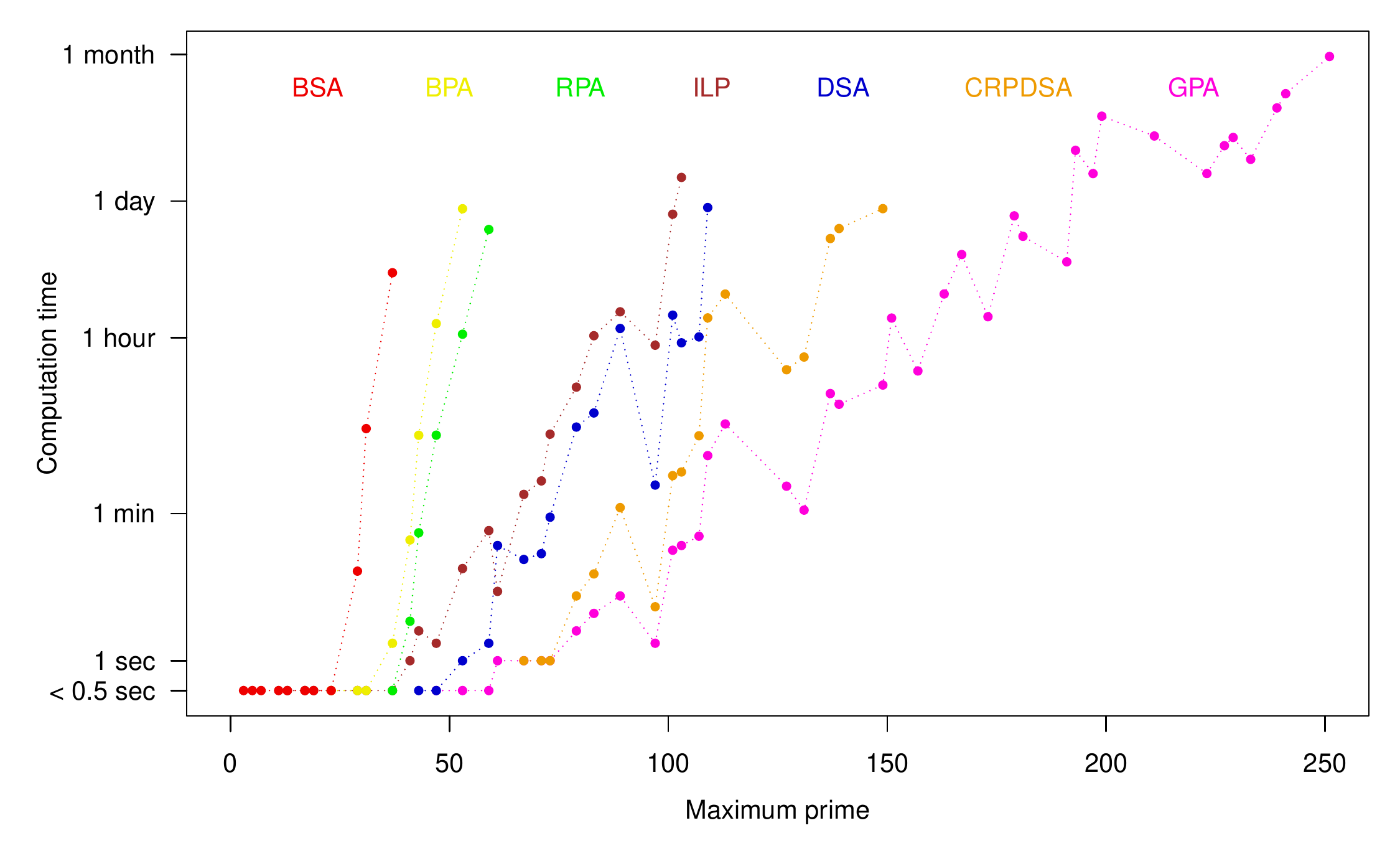}
  \caption{Comparison of the computational cost for various algorithms.}
\end{figure}

%%%%%%%%%%%%%%%%%%%%%%%%%%%%%%%%%%%%%%%%%%%%%%%%%%%%%%%
\subsection{Ancillary data}
%%%%%%%%%%%%%%%%%%%%%%%%%%%%%%%%%%%%%%%%%%%%%%%%%%%%%%%

In addition to this paper, we provide four files including the complete results of our calculations. The first file presents intrinsic data of the DSA algorithm \ref{DSA}.\\

\textbf{psi\_min.txt}\\

This file contains a table of values of the function $\psi_{min}(m,k)$ as described in definition \ref{psi_min} for $m\le500$ and $k\le11$. The data were computed using a brute force approach very similar to the Basic Sequential Algorithm \ref{BSA}. While BSA intends to maximise the number of covered positions in a given array, the computation of $\psi_{min}(m,k)$ needs to minimise it.\\

The other ancillary files contain exhaustive lists of all sequences of the \mbox{appropriate} maximum lengths.\ According to propositions \ref{remainders} and \ref{permutation}, these sequences can be\linebreak represented in three ways as always has been concluded in remark \ref{representation}.\\

\textbf{moduli.txt}\\

This file contains the modulus-representation of the sequences. In remark \ref{representation},\linebreak paragraph (1), every position $q\in\{1,\dots,m\}$ of the sequence was related to at least one prime modulus $p_i,\ i\in\{2,\dots,n\}$. The progression of primes $p_q,\ q\in\{1,\dots,m\}$ of the minimum appropriate moduli $p_q$ each position $q$ is a unique representation of the sequence under consideration.\\

\textbf{remainders.txt}\\

This file contains the remainder-representation of the sequences, i.e. the ordered set\linebreak of remainders $a_i\ (mod\ p_i),\ i\in\{2,\dots,n\}$ as described in remark \ref{representation}, paragraph (2).\\

\textbf{permutations.txt}\\

This file contains the permutation-representation of the sequences, i.e. the permutation of primes $\pi_i,\ i\in\{2,\dots,n\}$ as described in remark \ref{representation}, paragraph (3).\\

The sequences in \dq remainders.txt\dq\ are separately sorted  for each $n$ by ascending\linebreak remainders. This order was maintained in the other files \dq permutations.txt\dq\ and\linebreak \dq moduli.txt\dq\ to make a direct comparison possible.

%%%%%%%%%%%%%%%%%%%%%%%%%%%%%%%%%%%%%%%%%%%%%%%%%%%%%%%
\subsection{Final remarks}
%%%%%%%%%%%%%%%%%%%%%%%%%%%%%%%%%%%%%%%%%%%%%%%%%%%%%%%

All depicted algorithms may also be applied to arbitrary sets of different primes. The specific choice of consecutive primes was not necessarily required. With the help\linebreak of prime separation as described in remark \ref{properties}, all values of the original Jacobsthal function $j(n)$ can therefore be computed using a generalised implementation of one\linebreak of these algorithms.\\

\subsubsection*{Acknowledgement}
The authors would like to express their appreciation to the On-Line Encyclopedia of Integer Sequences \cite{OEIS}. It provides an exceptional collection of integer progressions. This data is very informative and helps discover new items and links within the wide field of number theory.

Furthermore, the work of the COIN-OR project \cite{COIN} is highly acknowledged.\linebreak Modern approaches in computational operations research are developed,\linebreak implemented, and disseminated in open source software.\\

\subsubsection*{Contact}
marioziller@arcor.de\\
axelmorack@live.com

\pagebreak

%%%%%%%%%%%%%%%%%%%%%%%%%%%%%%%%%%%%%%%%%%%%%%%%%%%%%%%
%%%%%%%%%%%%%%%%%%%%%%%%%%%%%%%%%%%%%%%%%%%%%%%%%%%%%%%
\bibliography{References}     
%%%%%%%%%%%%%%%%%%%%%%%%%%%%%%%%%%%%%%%%%%%%%%%%%%%%%%%
%%%%%%%%%%%%%%%%%%%%%%%%%%%%%%%%%%%%%%%%%%%%%%%%%%%%%%%

%%%%%%%%%%%%%%%%%%%%%%%%%%%%%%%%%%%%%%%%%%%%%%%%%%%%%%%
\end{document}